\documentclass[a4paper,11pt]{article}
\usepackage[utf8]{inputenc}
\usepackage{fullpage,url,amssymb,graphicx}
\usepackage[english]{babel}
\newcommand{\CC}{\mathbb{C}}

\newcommand{\FF}{\mathbb{F}}
\newcommand{\del}{\partial}
\newcommand{\Ann}{\mathrm{Ann}}
\newcommand{\mat}[4]{\left(\begin{array}{cc} #1 & #2 \\ #3 & #4 \end{array}\right)}
\newtheorem{theorem}{Theorem}
\newtheorem{prop}{Proposition}
\newtheorem{definition}{Definition}

\newtheorem{lemma}{Lemma}
\newtheorem{example}{Example}
\newsavebox{\BOXcochontriste}
\sbox{\BOXcochontriste}
{
  \includegraphics[width=1.5em]{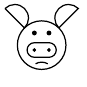}
}
\newcommand{\cochontriste}{\usebox{\BOXcochontriste}}
\newsavebox{\BOXcochonquirit}
\sbox{\BOXcochonquirit}
{
  \includegraphics[width=1.7em]{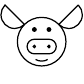}
}
\newcommand{\cochonquirit}{\usebox{\BOXcochonquirit}}

\newenvironment{proof}[0]
  {\cochontriste}
  {\ifvmode\leavevmode\fi\nolinebreak\cochonquirit\newline
}
\catcode`*=13
\def*{\*}

\begin{document}
\title{Symbolic Integration in Prime Characteristic}
\author{Bill Allombert}
\date{}

\maketitle

\begin{abstract}
In this paper we study elementary extensions of
differential fields in prime characteristic.
In particular, we show that, in contrast to Liouville's result in
characteristic zero, all elements of an elementary extension
admit an antiderivative in some logarithmic extension.
\end{abstract}

\section{Introduction}

\subsection{Motivation}

In a famous paper \cite{L}, Liouville proved that some elementary functions of
a real or complex variable do not admit elementary antiderivatives, the usual
example being $x\mapsto\exp(x^2)$.
In 1968, Risch gave an algorithm (\cite{Ri}) to decide if an elementary function
of a real or complex variable admits an elementary antiderivative.

The modern definition of elementary functions \cite{Ro} as elements of an elementary
differential extension of $\CC(X)$ can be generalised to any base field. In
this paper, we study the problem when the base field has finite characteristic.

In this paper we study differential fields in prime characteristic from the
point of view of symbolic integration. Below is an example of the kind
of phenomenon which interests us.

Let $K=\FF_p(X,E)$ and $\del$ the derivation defined by $\del(X)=1$ and
$\del(E)=2*X*E$.
The element $E$ satisfies the same first-order linear differential equation
than the complex function $f:x\mapsto\exp(x^2)$, namely $f'=2*x*f$, thus we can
see $E$ as a $\FF_p$-analogue of $f$.
However if we pick $p=3$ and set $y=\frac{1-X^2}{X^3}*E$,
or pick $p=5$ and set $y=\frac{X^4-2*X^2+2}{2*X^5}*E$,
a straightforward computation shows that in both case $\del(y)=E$,
while $f$ is known not to have an antiderivative.

The object of this paper is to study whether such formulae always exist
and how to compute them.

We will make use of iterated derivations in the naive sense, and not use Hasse
derivatives, even though we are in finite characteristic.  However the more
striking results we obtain can be stated without reference to iterated
derivations, but they will be play a crucial role in the proofs.

\section{Definitions}

Let $K$ be a field. A \emph{derivation} on $K$ is a group homomorphism
$\del=\del_K$ from $K$ to $K$ such that for all $a$, $b$ in $K$ the
following holds:
$$\del(a*b) = \del(a)*b+a*\del(b)\enspace.$$

The kernel of $\del_K$ is a field called the field of constants of
$(K,\del_K)$, and is denoted by $C(K)$.

A pair $(K,\del_K)$ is a \emph{differential field} if $K$ is a field and
$\del_K$ is a derivation on $K$.

A \emph{differential extension} of $(K,\del_K)$ is a differential field
$(L,\del_L)$ such that $L/K$ is a field extension and $\del_L|_K=\del_K$.

Let $p$ be a prime number and $(K,\del_K)$ be a differential field of
characteristic $p$. Derivations in characteristic $p$ have two specific
properties:

\begin{itemize}
\item for all $a\in K$, $\del(a^p)=0$.
\item for all positive integers $e$, $\del^{p^e}$ is a derivation.
\end{itemize}

The simplest non-trivial example of a differential field in characteristic $p$
is $(\FF_p(X),\del)$ where $\del(F)=F'$.

Furthermore the $p$-th derivative of an element of $\FF_p[X]$ is zero.
As a consequence, $X^{p-1}$ does not admit a polynomial antiderivative:
indeed, by Wilson's formula we have $\del^{p-1}(X^{p-1})=-1$,
so if $u$ was a polynomial antiderivative of $X^{p-1}$ we would have
$\del^p(u)=-1$, a contradiction.

A second issue is that inseparable algebraic extensions either cannot be
provided with a compatible derivation, or admit an infinite number of
compatible derivations.
Consider the fields $K=\FF_p(X)$ and $L=\FF_p(X^{1/p})$
(for the usual derivation).

The identity $\del_L(X)=\del({\left(X^{1/p}\right)}^p)=0$ is incompatible with
$\del_K(X)=1$. It follows that $L/K$ is an algebraic extension of fields
which cannot be extended to an extension of differential fields.

A third issue is that if $L/K$ is a transcendental extension of differential
fields, then $C(L)$ will be a transcendental extension of $C(K)$.

Thus, by contrast to characteristic $0$, it is not possible to require
that $C(L)=C(K)$. Instead, we could require that $C(L)=L^p*C(K)$, but
we will not need it.

\subsection{Elementary extensions}

\begin{definition}
A differential extension $(L,\del_L)$ of $(K,\del_K)$ is logarithmic
of type ``$log(u)$'' if there exists $u\in K^{\*}$ and $y\in L$ generating $L$
over $K$ such that $\del(u)=\del(y)*u$.
\end{definition}

\begin{definition}
A differential extension $(L,\del_L)$  of $(K,\del_K)$ is exponential
of type ``$exp(u)$'' if there exists $u\in K^{\*}$ and $y\in L$ generating
$L$ over $K$ such that $\del(y)=\del(u)*y$.
\end{definition}

\begin{definition}
A differential extension $(L,\del_L)$ of $(K,\del_K)$ is elementary if there
exists a tower of differential extensions
$L_0=K\subseteq L_1 \ldots\subseteq L_n = L$ such that
each extension $L_{i+1}/L_i$ is either algebraic separable, logarithmic,
or exponential.
\end{definition}

\subsection{Linear differential operators}

This sections lists some basic results about linear differential operators
in characteristic $p$ that will be useful.

Let $(K,\del_K)$ be a differential field. The map $\del_K$ is an endomorphism
of the $C(K)$-vector space $K$. If $P$ is a polynomial with coefficients in
$C(K)$, we denote by $P(\del_K)$ the application of $P$ to $\del_K$ in
the algebra of $C(K)$-endomorphisms of $K$.

Let $y$ be an element of $K$. The annihilator of $y$, $\Ann(y)=\{P\in C(K)[X]|
P(\del)(y)=0\}$, is an ideal of the polynomial ring $C(K)[X]$.

The element $y\in K$ satisfies a linear differential equation with constant
coefficients if and only if $\Ann(y)\neq\{0\}$.

\begin{definition}[$p$-polynomials]
(See \cite[page 234]{J}.)
A polynomial $P$ over a field $K$ is a $p$-polynomial if it can be written as
$P=\sum_{j=0}^n a_j*X^{p^j}$ with $a_j\in K$ for $0\leq j\leq n$.
\end{definition}

\begin{lemma}
If $K$ is a field and $I$ a non-zero ideal of $K[X]$, then $I$ contains a
non zero $p$-polynomial.
\end{lemma}
\begin{proof}
Since $I$ is non-zero, the quotient $K[X]/I$ is a finite dimensional
$K$-vector space.
If $\pi$ denotes the projection from $K[X]$ to $K[X]/I$, it follows that
the infinite family $(\pi(X^{p^n}))_{n\geq 0}$ is $K$-linearly dependent. Thus
there exist elements $(a_j)_{j=1}^k$ of $K$, not all zero, such that
$\sum_{j=1}^k a_j*X^{p^j}\in I$.
\end{proof}

\begin{lemma}\label{peqn}
Let $(K,\del_K)$ be a differential field, and $(y_i)_{i=1}^n$ be a family of
elements of $K$. If the $y_i$ satisfy a linear differential equation with
constant coefficients for all $1\leq i\leq n$, then there exists a non-zero
$p$-polynomial $P$ in $C(K)[X]$ such that $P(\del)(y_i)=0$ for all
$1\leq i\leq n$.
\end{lemma}

\begin{proof}
Let $I=\bigcap_{i=1}^n \Ann(y_i)$. Since $\Ann(y_i)\neq\{0\}$ for all
$1\leq i\leq n$ and $C(K)[X]$ is a domain, $I\neq\{0\}$.
From the lemma, it follows that there exists a non-zero $p$-polynomial $P$ in
$C(K)$ such that $P\in I$, so in particular $P(y_i)=0$ for all $1\leq i\leq n$.
\end{proof}

\begin{prop}\label{derpk}
If $(K,\del)$ is a differential field of characteristic $p>0$ and $P$ is a
$p$-polynomial in $C(K)$, then the $C(K)$-endomorphism $P(\del)$ is a
derivation which commutes with $\del$.
\end{prop}
\begin{proof}
This follows from the fact that for all $j\geq 1$, $\del^{p^j}$ is a
derivation.
\end{proof}

We will make use of the following notation.
\begin{lemma}\label{Qdel}
Let $(K,\del)$ be a differential field and
$Q=\sum_{i=0}^n a_i*X_i\in K[X]$ be a polynomial.
We set
$$Q^{\del} = \sum_{i=0}^n \del_K(a_i)*X^i\enspace.$$
If $u\in K$ then $$\del(Q(u))=\del(u)*Q'(u)+Q^{\del}(u)\enspace.$$
\end{lemma}

\section{Linear differential fields}

\subsection{Linear fields}

\begin{definition}
Let $(K,\del_K)$ be a differential field, let $y$ be an element of $K$, and
let $C$ be a sub-differential field of $K$.
We will say that $y\in K$ satisfies a linear differential equation with
coefficients in $C$ if there exists a non-zero polynomial $P\in C[X]$ such
that $P(\del)(y)=0$.
\end{definition}

\begin{definition}
A differential field $(K,\del_K)$ is linear if every element satisfies a
linear differential equation with coefficients in $C(K)$.
\end{definition}

\subsection{Linear extensions}

This section establishes that some common type of extensions of
linear differential fields are linear.

\begin{prop}
Let $(K,\del_K)$ be a differential field.
The set $F$ of elements of $K$ that satisfy a
linear differential equation with coefficients in $C(K)$
is a sub-differential field of $K$.
\end{prop}
\begin{proof}
Let $u$ and $v$ be two elements of $F$. By Lemma~\ref{peqn},
there exists a $p$-polynomial $P\in C(K)[X]$ such that $P(\del)(u)=0$ and
$P(\del)(v)=0$. Since $P(\del)$ is a derivation, $\ker P(\del)$ is a field,
so $P(\del)(u+v)=0$, $P(\del)(u*v)=0$, and if $u\neq 0$, $P(\del)(u^{-1})=0$
and $P(\del)(\del(u))=\del(P(\del)(u))=0$.
We conclude that $F$ is a differential field.
\end{proof}

\begin{lemma}\label{linfield}
Let $(K,\del_K)$ be a linear differential field, and let
$(L,\del_L)$ be a differential extension of $(K,\del_K)$.
If $L/K$ is algebraic separable, then $(L,\del_L)$ is linear.
\end{lemma}
\begin{proof}
Let $y$ be an element of $L$ and $U=\sum_{i=0}^n u_i*X^i$ be its minimal
polynomial over $K$.
By Lemma~\ref{peqn}, there exists a non-zero $p$-polynomial $P$ in $C(K)[X]$
such that $P(\del)(u_i)=0$ for all $0\leq i \leq n$.
Since $P(\del)$ is a derivation,
$$\sum_{i=0}^n P(\del)(u_i*y^i)=\sum_{i=0}^n P(\del)(u_i)*y^i+\sum_{i=0}^n u_i*i*P(\del)(y)*y^{i-1}\enspace,$$
so
$$\sum_{i=0}^n P(\del)(u_i*y^i)=P(\del)(y)*U'(y)\enspace.$$
Since $U(y)=0$, it follows that $\sum_{i=0}^n P(\del)(u_i*y^i)=0$, and so
$P(\del)(y)*U'(y)=0$.
Since $y$ is separable, $U'(y)\neq 0$, so $P(\del)(y)=0$.
\end{proof}

\begin{lemma}\label{linexp}
Let $(K,\del_K)$ be a linear differential field, and let $(L,\del_L)$ be a
differential extension of $(K,\del_K)$. If $L/K$ is of exponential type
then $(L,\del_L)$ is linear.
\end{lemma}
\begin{proof}
The field $L$ is generated by an element $y$ such that there exists
$u\in K^{\*}$ with $\del(y)=\del(u)/u$, or equivalently $u*\del(y)=\del(u)$.
This is a classical differential equation,
whose derivatives are given by the multivariate Bell polynomials.
In particular in characteristic $p$, Carlitz~\cite{C} Formula~(1.4)
gives the equation
\begin{equation}
\del^{p^r}(y)=(\sum_{i=0}^r {\del^{p^i}(u)}^{p^{r-i}})*y\enspace.
\end{equation}
We note that $$c_r=\sum_{i=0}^{r-1} {\del^{p^i}(u)}^{p^{r-i}})$$ belongs to
$C(K)$, and we rewrite the equation as
$$\del^{p^r}(y)= (\del^{p^r}(u) +c_r)*y\enspace.$$

By Lemma~\ref{peqn}, there exists a family $(a_j)_{j=1}^k$ of
elements of $C(K)$ not all equal to $0$ such that
$\sum_{j=0}^k a_j*\del^{p^k}(u)=0$.
It follows that
\begin{eqnarray}
\sum_{j=0}^k a_j*\del^{p^j}(y)&=&(\sum_{j=0}^k a_j*\del^{p^r}(u))*y+(\sum_{j=0}^k a_j*c_j)*y\\
                               &=&(\sum_{j=0}^k a_j*c_j)*y\enspace.
\end{eqnarray}
Since the $a_j$ are not all equal to $0$ and $\sum_{j=0}^k a_j*c_j$ belongs to
$C(K)$, $y$ satisfies a linear differential equation with coefficients in
$C(K)$, so by Theorem~\ref{linfield} $L$ is linear.
\end{proof}

\begin{lemma}
Let $(K,\del_K)$ be a linear differential field, and let $(L,\del_L)$ be a
differential extension of $(K,\del_K)$. If $L/K$ is of logarithmic type
then $(L,\del_L)$ is linear.
\end{lemma}
\begin{proof}
There exists a generator $y$ of $L$ and $u\in K^{\*}$
such that $\del(y)=\del(u)/u$, or equivalently $u*\del(y)=\del(u)$.
By applying Carlitz's formula, we have
\begin{equation}
\del^{p^r}(u)=(\sum_{i=0}^r (\del^{p^i}(y))^{p^{r-i}})*u\enspace.
\end{equation}
We note that $c_r=\sum_{i=0}^{r-1} (\del^{p^i}(y))^{p^{r-i}})$ belongs to
$C(L)$ and we rewrite the equation as $\del^{p^r}(u)= (\del^{p^r}(y) +c_r)*u$.
By Lemma~\ref{peqn}, there exists a family $(a_j)_{j=1}^k$ of
elements of $C(K)$ not all equal to $0$ such that
$$\sum_{j=0}^k a_j*\del^{p^j}(u)=0\enspace.$$
  From
\begin{eqnarray}
\sum_{j=0}^k a_j*\del^{p^j}(u)&=&(\sum_{j=0}^k a_j*\del^{p^j}(y)+\sum_{j=0}^k a_j*c_j)*u
\end{eqnarray}
it follows that $$\sum_{j=0}^k a_j*\del^{p^j}(y)+\sum_{j=0}^k a_j*c_j=0\enspace,$$
and since $\sum_{j=0}^k a_j*c_j$ belongs to $C(L)$,
$\sum_{j=1}^k a_j*\del^{p^j+1}(y)=0$.
Since the $a_j$ are not all equal to $0$, $y$ satisfies a linear differential
equation with coefficients in $C(K)$, so by Theorem~\ref{linfield} $L$ is
linear.
\end{proof}

We have thus proved:

\begin{theorem}\label{diffeqn}
Let $(K,\del_K)$ be a linear differential field, and let $(L,\del_L)$ be a
differential extension of $(K,\del_K)$. If $L/K$ is elementary,
then $(L,\del_L)$ is linear.
\end{theorem}

The following is a generalization of Lemma~\ref{linexp}.

\begin{theorem}\label{KtoCL}
Let $(K,\del_K)$ be a linear differential field, and $(L,\del_L)$ a
differential extension of $(K,\del_K)$. If $y\in L$ statisfies
a linear differential with coefficients in $K$ then it satisfies a
linear differential equation with coefficients in $C(K)$.
\end{theorem}
\begin{proof}
Let $u_0,\ldots,u_n$ in $K$ with $u_n\neq 0$ be such that
$$u_0*y+u_1*\del_L(y)+u_2*\del_L(\del_L(y))\ldots+u_n*\del_L^n(y)=0\enspace.$$
By Lemma~\ref{peqn}, there exists a non-zero $p$-polynomial $P\in C(K)[X]$
such that $P(\del_K)(u_i)=0$ for all $0\leq i \leq n$.
Since $P(\del_K)$ is a derivation, the $C(K)$ vector space $F=\ker P(\del_K)$
is a subfield of $K$ that contains $(u_i)_{i=0}^n$ and with constant field
$C(F)=C(K)$.

Let $E$ be the sub-$F$-vector space of $L$ spanned by $(\del^k(y))_{k=0}^{n-1}$.
It follows from the differential equation that $E$ is stable by $\del_L$, and
the restriction $D$ of $P(\del_L)$ to $E$ is an $F$-endomorphism of $E$.
Let $Q\in F[X]$ be the monic minimal polynomial of $D$.
If $v\in E$ then $Q(D)(v)=0$, and so $\del_L(Q(D)(v))=0$.
Since $\del$ and $D$ commute,
$$\del_L(Q(D)(v))=Q^{\del}(D)(v)+Q(D)(\del_L(v))\enspace.$$
Since $Q(D)=0$, it follows that $Q^{\del}(D)(v)=0$ for all $v\in E$.
Since $Q$ is the minimal polynomial of $D$ and the degree of $Q^{\del}$
is strictly less than the degree of $Q$ (since $Q$ is monic), it follows
that $Q^\del=0$, so $Q$ belongs to $C(K)[X]$.
In particular it follows that $Q(P)$ is a non-zero element
of $C(K)[X]$ which satisfies $(Q(P))(\del_L)(y)=0$.
\end{proof}

\begin{example}
Let $K=(\FF_3(X),\del_K)$ be such that $\del_K(X)=1$,
and let $L$ be the differential extension $(\FF_3(X,Y,Y_1),\del_L)$ where
$\del_L(X)=1$, $\del_L(Y)=Y_1$ and $\del_L(Y_1)=X*Y$.

The element $Y$ satisfies the linear equation $\del^2(Y)=X*Y$
(the Airy~\cite{A} equation)
with coefficients in $F=\FF_3(X)$ and $\del^3(X)=0$.
By applying $\del$, it follows that $\del^3(Y)=Y+X*Y_1$ and
$\del^3{Y_1}=\del^4(Y)=2*Y_1+X^2*Y$.
Thus the matrix of $\del^3$ in the basis $(Y,Y_1)$ is
$\mat{1}{X^2}{X}{2}$, whose minimal polynomial is $P(T)=T^2-X^3-1$.
Thus $Y$ satisfies the equation $\del^6(Y)=(X^3+1)*Y$ with coefficients in
$C(F)=\FF_3(X^3)$.
\end{example}

\section{Integration}

\begin{lemma}
Let $(K,\del)$ be a differential field of characteristic $p>0$.
If $u\in K$ is such that $\del^{p^k}(u)\in C(K)^{\times}$ and $y\in K$ is such
that $\del(y)=\del(u)/u$, then $\del^{p^{k+1}}(y)\in C(K)^{\times}$.
\end{lemma}

\begin{proof}
By applying again Carlitz's formula to the identity $\del(u)=\del(y)*u$,
we obtain
\begin{equation}
\del^{p^k}(u) = \left(\sum_{i=0}^k {\del^{p^i}(y)}^{p^{k-i}}\right)*u\enspace.
\end{equation}
After setting $D=\del^{p^k}$ and
$c_k=\sum_{i=0}^{k-1} {\del^{p^i}(y)}^{p^{k-i}}\in C(K)$,
the equation reads
\begin{equation}
D(u)/u = D(y) + c_k\enspace.
\end{equation}
Since the endomorphism $D$ is a derivation of $K$ whose kernel contains $C(K)$,
if $F$ is a polynomial with coefficients in $C(K)$ then
$D(F(u))=D(u)*F'(u)$. Since the $p-1$-st derivative of the polynomial
$X^{p-1}$ is equal to $-1$ by Wilson's theorem,
we obtain the formula: $D^{p-1}(u^{p-1})=-D(u)^{p-1}$.
Noting that $D^{p-1}=\del^{p^{k+1}-p^k}$, it follows that
$-D(u)^p/u^p = \del^{p^{k+1}}(y)$, and since $D(u)\in C(K)^{\times}$ we
have proved that $\del^{p^{k+1}}(y)\in C(K)^{\times}$.
\end{proof}

\begin{prop}\label{loglogp}
Let $(K,\del)$ be a differential field of characteristic $p>0$,
$n\geq 0$ an integer, and $u\in K$ such that $\del^{p^n}(u)=1$.
There exists a logarithmic extension $L/K$ of type $\log(u)$
and $y\in L$ such that $\del^{p^{n+1}}(y)=1$.
\end{prop}
\begin{proof}
We denote by $L$ the field of rational functions $K(z)$.
We extend the derivation $\del_K$ to $L$ by setting $\del(z)=\del(u)/u$.
The extension $L/K$ is logarithmic of type $\log(u)$.
It follows from the lemma that $c=\del^{p^{k+1}}(z)\in C(L)^{\times}$,
so by setting $y=z/c$, it follows that $\del^{p^{k+1}}(z/c)=1$.
\end{proof}

It is possible to prove a more precise result:

\begin{prop}\label{loglogp2}
Let $(K,\del)$ be a differential field of characteristic $p>0$,
$n\geq 0$ an integer, and $u\in K$ such that $\del^{p^n}(u)=1$.
Then either there exists $y\in K$ such that $\del^{p^{n+1}}(y)=1$, or
there exists a differential extension $L/K$ of logarithmic type
such that $C(L)=C(K)*L^p$ and $y\in L$ such that $\del^{p^{n+1}}(y)=1$.
\end{prop}
\begin{proof}
We assume that the equation $\del^{p^{n+1}}(y)=1$ has no solution $y\in K$.
By the lemma, this implies that the equation $\del(z)=\del(u)/u$ has no
solution $z\in K$.
We build $L=K(z)$ as in Proposition~\ref{loglogp}.
It only remains to prove that $C(L)=C(K)*L^p$.
An element of $C(L)$ can be written as $f(z)/g(z)$ where $f$ and $g$ are
two elements of $K[X]$. We set $Q=f*g^{p-1}\in K[X]$.
Since $f(z)/g(z) = Q(z)/g(z)^p$, it follows that $Q\in C(L)$ and that
proving $Q(z)\in C(K)*L^p$ will prove that $f(z)/g(z)\in C(K)*L^p$.

We write $Q=\sum_{i=0}^{m} a_i*X_i$ with $a_i\in K$.
Since $\del(Q(y))=0$, it follows that
$$\del(u)*Q'(y)+u*Q^{\del}(y)=0\enspace.$$
By identifying coefficients we find the equations
\begin{eqnarray}
u*\del(a_m)&=& 0 \\
\del(u)*(i+1)*a_{i+1} + u*\del(a_i) = 0
\end{eqnarray}
By using the fact that the equation $\del(u)*c = u*\del(z)$ with unknowns $c\in
C(K)$ and $z\in K$ only admits the solution $c=0$, $z=0$, it follows easily by
induction that $Q\in C(K)[X^p]$, so $Q(y)\in C(K)(y^p)=C(K)*L^p$.
\end{proof}

\begin{prop}\label{loglog}
Let $(K,\del)$ be a differential field of characteristic $p>0$,
$n$ a positive integer, and $u\in K$ such that $\del^{n}(u)=1$.
Then either there exists $y\in K$ such that $\del^{n+1}(y)=1$, or
there exists a logarithmic extension $L/K$ such that $C(L)=C(K)*L^p$
and $y\in L$ such that $\del^{n+1}(y)=1$.
\end{prop}
\begin{proof}
Write $n=p^k+l$ with $l\geq 0$ and $k$ as large as possible, and assume
that the equation $\del^{n+1}(y)=1$ has no solution in $K$.
This implies that the equation $\del^{p^{k+1}}(y)=1$ has no solution in $K$.
By Proposition~\ref{loglogp2}, we build $L$ satisfying the above conditions
and $z$ such that $\del^{p^{k+1}}(z)=1$.
We set $y=\del^{p^{k+1}-(n+1)}(z)$ and it follows that $\del^{n+1}(y)=1$.
\end{proof}

\section{Antiderivable fields}

\begin{definition}
A differential field $(K,\del)$ is \emph{antiderivable} if it is linear and
such that $\del$ takes the value $1$ on $K$.
\end{definition}

\begin{example}
The differential field $(\FF_p(X),\del)$ where $\del$ is the standard
derivation is antiderivable.

Indeed, every element $u$ satisfies $\del^p(u)=0$, and $\del(X)=1$.
\end{example}

\begin{prop}
An elementary extension of an antiderivable differential field is antiderivable.
\end{prop}
\begin{proof}
This follows from Theorem~\ref{diffeqn} and the fact that if $\del$ takes the
value $1$ over $K$, it takes it a fortiori over any larger field.
\end{proof}

\begin{lemma}\label{logpol}
Let $(K,\del)$ be a differential field, and $u$ and $v$ be two elements of $K$.
If there exists an integer $n\geq 1$ such that $\del^n(u)=0$ and $\del^{n-1}(v)=1$,
then $u$ belongs to the vector space generated by $(\del^k(v))_{k=0}^{n-1}$ over $C(K)$.
\end{lemma}
\begin{proof}
This is a classical result. If $n=1$, then $v=1$ and $u\in C(K)$, so the
result is true. Otherwise assume by induction that the result is
true for $n$. Let $u$ and $v$ be such that $\del^{n+1}(u)=0$ and $\del^n(v)=1$.
It follows that $\del(u)$ and $\del(v)$ satisfy $\del^n(\del(u))=0$ and
$\del^{n-1}(\del(v))=1$.
By the induction hypothesis, there exist elements $(c_i)_{k=0}^{n-1}$ in $C(K)$
and not all zero, such that
$\del(u)=\sum_{k=0}^{n-1} c_i*\del^k(\del(v))=0$.
By setting $S=\sum_{k=0}^{n-1} c_i*\del^k(v)$, it follows that $c_n=u-S$
belongs to $C(K)$. Since $\del^n(v)=1$ we have
$u=\sum_{k=0}^{n} c_i*\del^k(v)$, which concludes the proof.
\end{proof}

The following result justifies the definition:

\begin{theorem}\label{integ}
Let $(K,\del)$ be an antiderivable differential field.
Every element $u\in K$ admits an antiderivative in an extension of logarithmic
type.
\end{theorem}
\begin{proof}
By hypothesis, $u$ satisfies a linear differential equation with coefficients
in $C(K)$. Thus there exist $n\geq 0$ and $(a_i)_{i=1}^k$ in $C(K)$ such that
$\del^n(u)=\sum_{i=1}^k a_i*\del^{n+i}(u)$.
Set $v=\sum_{i=1}^k a_i*\del^{i-1}(u)$ and $w=\del(v)-u$, so that
$\del^n(w)=0$.
If $n=0$ then $v$ is an antiderivative of $u$.
Otherwise, we may assume that $n$ is minimal for the given $(a_i)_{i=1}^k$,
in other words that $\del^{n-1}(w)\neq 0$, and so $\del^{n-1}(w)$ is a
non-zero constant in $C(K)$. By Proposition~\ref{loglog},
there exist a logarithmic extension $L$ and an element $z\in L$ such
that $\del^n(z)=1$. Since $\del^n(w)=0$,
Lemma~\ref{logpol} implies that $w=\sum_{k=0}^{n-1} c_i*\del^k(\del(z))$,
which leads to
$u=\del\left(v-\sum_{k=0}^{n-1} c_i*\del^k(w)\right)$.
Thus $u$ admits an antiderivative.
\end{proof}

We have the slightly stronger statement:
\begin{theorem}
Let $(K,\del)$ be an antiderivable differential field, and $u\in K$.
Either $u$ admits an antiderivative in $K$, or it admits
an antiderivative in a logarithmic extension $L$ such that
$C(L)=C(K)*L^p$.
\end{theorem}

\begin{example}
We look for the characterictic $p$ analogue of the
antiderivative of $exp(exp(x))$.

Consider the differential field $K=(\FF_p(E,F),\del)$,
where $\del(E)=E$ and $\del(F)=E*F$.
  From the obvious equation $\del^p(E)-\del(E)=0$ and Lemma~\ref{linexp}, we
obtain the equation $\del^p(F)-\del(F)=E^p*F$. We conclude that
$\frac{\del^{p-1}(F)-F}{E^p}$ is an antiderivative of $F$.
\end{example}

\begin{example}
Consider the differential field $K=(\FF_p(X),\del_K)$ ,
where $\del_K(X)=1$, and we look for an antiderivative of $X^{p-1}$.
The differential field $L=(\FF_p(X,Y),\del_L)$, where $\del_L(X)=1$ and $\del_L(Y)=1/X$, is a logarithmic extension of $K$ and we see that
$\del_L(Y*X^p)=X^{p-1}$.
\end{example}

\begin{theorem}
Let $(L,\del)$ be an elementary extension of $(\FF_p(X),\del)$ where $\del$ is
the standard derivation. Every element $u\in L$ admits an
antiderivative in some elementary differential extension of $(\FF_p(X),\del)$.
\end{theorem}

\section{Linear differential equations}

\begin{theorem}
Let $(K,\del)$ be an antiderivable differential field and $Q\in C(K)[X]$.
The linear differential equation $Q(\del)(y)=0$ admits a solution in
some elementary extension of $K$.
\end{theorem}
\begin{proof}
The proof extends Euler's method of looking for a solution of the form
$\exp(\alpha*x)$ for equations with inseparable characteristic polynomial.

Assuming $Q\neq 0$, we write $Q=P(X^{p^e})$ with $e$ as large as
possible, so that $P'\neq 0$.
Since $K$ is antiderivable, there exists $u$ in some elementary extension $L_1$
of $K$ such that $\del_{L_1}^{p^e}(u)=1$. We consider a separable extension
$L_2=L_1(\alpha)$ of $L_1$ such that $\del(\alpha)=0$, and the exponential
extension $L_3/L_2$ generated by $E$ such that
$\del_{L_3}(E)=\alpha\del_{L_2}(u)*E$.
By again applying Carlitz's formula we obtain
\begin{equation}
\del^{p^e}(E) = \left(\sum_{i=0}^e {\alpha^{p^{e-i}}*{\del^{p^i}(u)}^{p^{e-i}}}\right)*E\enspace.
\end{equation}
The polynomial $A(X)=\sum_{i=0}^e {X^{p^{e-i}}*{\del^{p^i}(u)}^{p^{e-i}}}$
is equal to $\sum_{i=0}^{e-1} {X^{p^{k-i}}*{\del^{p^i}(u)}^{p^{e-i}}} + X $
so belongs to $C(L_1)[X]$. Furthermore, $A'=1$ and $\del^{p^e}(E)=A(\alpha)*E$.
Thus $Q(\del)(E)=P(A(\alpha))*E$.
Setting $R=P(A(X))$ we have that $R'=P'(A(X))$ is non-zero since $P'$ is
non-zero. Thus $R$ admits a root $\alpha$ in some separable extension $L_2$ of
$L_1$, so it follows that $Q(\del)(E)=0$.
Note that the degree of $R$ is equal to the degree of $Q$.
\end{proof}

To conclude we give a partial proof of the result below:

\begin{theorem}
Let $(K,\del)$ be an antiderivable differential field and $P\in K[X]$.
The linear differential equation $P(\del)(y)=0$ admits a solution in
some elementary extension of $K$.
\end{theorem}
\begin{proof}
Without loss of generality, we can assume $P$ to be monic, and we write
$$P = \sum_{i=0}^n p_i*X^i\enspace.$$
Let $L$ be the differential field $(K(Y_0,\ldots,Y_{n-1}),\del)$ where
$\del(Y_i)=Y_{i+1}$ if $i<n-1$ and $\del(Y_n)=-\sum_{i=0}^{n-1} p_i*Y_i$.
We see that $P(\del)(Y)=0$.

Since $Y$ satisfies a differential equation with coefficients in $K$,
by Theorem~\ref{KtoCL}, it satisfies a linear differential
equation $Q(\del)(Y)=0$ with $Q\in C(K)[X]$.

By using the Euclidean division in the ring of differential operators $K[X,\delta]$
it is possible to construct a differential polynomial $R\in K[X,\delta]$ such that
if $Q(\del)(f)=0$, then $P(\del)(R(f,\del))=0$.
By the previous theorem, the equation $Q(\del)(f)=0$  admits a solution in some
elementary extension of $K$. Assuming the solution is sufficiently generic, the
element $R(f,\del)$ will be a non zero-solution to $P(\del)(y)=0$.
\end{proof}

\noindent Bill Allombert\\
CNRS, Institut de Math\'ematiques de Bordeaux, UMR 5251,\\
Universit\'e de Bordeaux, 351 cours de la Lib\'eration, F-33405 Talence Cedex, France\\
email: bill.allombert@math.u-bordeaux.fr

\end{document}